\documentclass{article}

\usepackage{arxiv}
\usepackage{amsmath,amssymb,amsthm,mathtools,subcaption}
\usepackage[utf8]{inputenc} 
\usepackage[T1]{fontenc}    
\usepackage{hyperref}       
\usepackage{url}            
\usepackage{booktabs}       
\usepackage{amsfonts}       
\usepackage{nicefrac}       
\usepackage{microtype}      
\usepackage{lipsum}
\usepackage{graphicx}
\usepackage[noabbrev]{cleveref}

\graphicspath{ {./images/} }

\newcommand{\map}[2]{\,{:}\,#1\!\longrightarrow\!#2}

\newcommand{\R}{\mathbb{R}}
\newcommand{\C}{\mathbb{C}}
\title{A new Evans function for quasi-periodic solutions of the linearised sine-Gordon equation}
\newtheorem{Theorem}{Theorem}
\theoremstyle{definition}

\newtheorem{corollary}[Theorem]{Corollary}

\author{
	W.A. Clarke \\
	School of Mathematics and Statistics\\
	University of Sydney\\
	\texttt{wcla7359@uni.sydney.edu.au} \\
	\And
	R. Marangell \\
	School of Mathematics and Statistics\\
	University of Sydney\\
	\texttt{robert.marangell@sydney.edu.au} \\
}

\begin{document}
	\maketitle
	\begin{abstract}
		We construct a new Evans function for quasi-periodic solutions to the linearisation of the sine-Gordon equation about a periodic travelling wave. This Evans function is written in terms of fundamental solutions to a Hill's equation. Applying the Evans-Krein function theory of \cite{KM} to our Evans function, we provide a new method for computing the Krein signatures of simple characteristic values of the linearised sine-Gordon equation. By varying the Floquet exponent parametrising the quasi-periodic solutions, we compute the linearised spectra of periodic travelling wave solutions of the sine-Gordon equation and track dynamical Hamiltonian-Hopf bifurcations via the Krein signature. Finally, we show that our new Evans function can be readily applied to the general case of the nonlinear Klein-Gordon equation with a non-periodic potential.
	\end{abstract}


	\section{Introduction}
	We consider the sine-Gordon equation,
	\begin{align}
	u_{tt} - u_{xx} + \sin(u) = 0, \label{eq:sG}
	\end{align}
	where $ u(x,t) \map{\R \times [0,+\infty)}{\R} $. This equation has been used in the modelling of a number of different physical and biological systems. For example, \cref{eq:sG} models the electrodynamics of a long Josephson junction arising in the theory of superconductors \cite{BP,DDKS}. Solitary wave solutions to \cref{eq:sG} have been used to describe the dynamics of DNA as it interacts with RNA-polymerase \cite{DG}. More recent research has seen sine-Gordon solitons used as scalar gravitational fields in the theory of general relativity, the solutions of which are solitonic stars and black holes \cite{CFMT}. \Cref{eq:sG} can also be derived from classical mechanics applied to a mechanical transmission line in which pendula are coupled to their nearest neighbours by springs obeying Hooke's law \cite{Knobel}. See \cite{BEM} for an extensive list of applications of the sine-Gordon equation.\par
	In this paper, we focus on the problem of spectral stability of periodic solutions to the sine-Gordon equation. In \cite{Scott}, Scott correctly attributed spectral stability and instability to various types of periodic wavetrains, however this was rigorously proved only recently in \cite{Jones13} in which the authors related the sine-Gordon equation to a Hill's equation using Floquet theory and known results on Hill's equation \cite{Jones13, MW}. Spectral stability plays an important role in the determination of the orbital stability of periodic travelling wave solutions, and this is a current topic of investigation. Under co-periodic perturbations, spectral instability of a periodic solution implies orbital instability in several contexts: for subluminal, librational solutions to the sine-Gordon equation \cite{Natali2011}; and for subluminal solutions to a nonlinear Klein-Gordon equation \cite{PN2014}. For the nonlinear Klein-Gordon equation with periodic potential, subluminal rotational waves are orbitally stable with respect to co-periodic perturbations \cite{PP2016}. \par
	There are a number of approaches taken in order to determine the spectral stability of travelling waves; in \cite{Stanislavova}, Stanislavova and Stefanov developed a stability index for travelling wave solutions to second order in time PDEs (including the Klein-Gordon-Zakharov system). The Evans function is another tool that has been used to investigate spectral stability of travelling waves. Evans used this eponymous function while studying the equations governing electrical pulses in nerve axons \cite{Evans}. Jones was the first to coin the term \textit{Evans function} in \cite{Jones84}, where he used this function to prove the stability of travelling wave solutions to the Fitzhugh-Nagumo equations close to a singular limit of the equations. Alexander, Gardner and Jones in \cite{AGJ1990}, and Gardner in \cite{Gar1993,Gar1997}, developed much of the theory and approach we use in this paper to construct and apply a periodic Evans function to problems of spectral stability. Evans function theory has been applied to a general nonlinear Klein-Gordon equation \cite{Lafortune}, and also more specifically to a perturbed sine-Gordon equation \cite{DDGV}. In the previous cases, the Evans functions used in the analyses of nonlinear Klein-Gordon type equations comes from a classical construction of the Evans function based on the exterior product of solutions. In \cite{Jones14}, the authors instead use Floquet theory to construct a periodic Evans function for quasi-periodic solutions of the nonlinear Klein-Gordon equation.\par
	This paper combines results for Hill's equation in \cite{Jones13} and the periodic Evans function from \cite{Jones14} in order to construct a new Evans function for quasi-periodic solutions to the linearised sine-Gordon equation. We show how this new Evans function can be used to calculate the Krein signature - a stability index that is used to detect Hamiltonian-Hopf bifurcations. In particular, we apply the method of \cite{KM} to adapt our Evans function to an Evans-Krein function, the derivatives of which can be used to calculate Krein signatures. There are a number of different methods of calculating the Krein signatures of purely imaginary eigenvalues; Kollar, Deconinck and Trichtchenko use the dispersion relation of the linearisation of a Hamiltonian PDE to calculate the Krein signatures of colliding eigenvalues on the imaginary axis, applying this to small-amplitude periodic travelling wave solutions to the generalized KdV equation \cite{KDT2019} and the Kawahara equation \cite{TDK2018}. H\v{a}r\v{a}gu{\c{s}} and Kapitula also use the dispersion relation to deduce the spectral stability of periodic wave solutions to NLS-type equations \cite{HK2008}. In the case of the generalized KdV equation, Bronski, Johnson and Kapitula used geometric information from the linearised ODE in the calculation of an instability criterion which takes into account the number of potential instabilities resulting from imaginary eigenvalues with negative Krein signature \cite{BJK2011}. In \cite{Kap2010}, Kapitula constructs the Krein matrix, a meromorphic function of the spectral parameter which can be used to calculate the Krein signatures of eigenvalues with non-zero imaginary part or real eigenvalues with negative Krein signature. Kapitula, Kevrekedis and Yan have applied the Krein matrix to the study of Bose-Einstein condensates governed by the Gross-Pitaevskii equation \cite{KKY2013}, and Bronski, Johnson and Kollar have used the Krein matrix in their development of an instability index theory for quadratic operator pencils \cite{BJK2014}.\par
	Our new Evans function leverages the simplicity of Hill's equation, which translates into a more elegant calculation of Krein signatures when compared to the Evans function of \cite{Jones14} and allows for tracking of Hamiltonian-Hopf bifurcations in terms of the Floquet exponent. This method distinguishes itself from the stationary methods of \cite{Jones13, Marangell2015} which instead calculate the zeroes of bespoke functions in order to detect the regions of the spectrum that exhibit Hamiltonian-Hopf instabilities. We focus firstly on the sine-Gordon equation before extending our results to the nonlinear Klein-Gordon equation with $ C^{2} $ potential. This is because we believe it is more illustrative to work through the specific example of the sine-Gordon equation which has both librational and rotational solutions. The results for sine-Gordon are well-established in the literature, and provide a check on the correctness of our numerical results.
	\subsection{Set-up of the eigenvalue problem}
	In travelling wave coordinates $ z = x - ct, \tau = t $, a travelling wave solution $  \hat{u}(z,\tau) $ to \cref{eq:sG} satisfies:
	\begin{align}
	(c^2 - 1)\hat{u}_{zz} - 2c\hat{u}_{z\tau} + \hat{u}_{\tau \tau} + \sin(\hat{u}) = 0. \label{eq:travellingsg}
	\end{align}
	A standing wave solution $ \hat{U}(z) $ in travelling wave coordinates will be independent of $ \tau $ and hence satisfies:
	\begin{align}
	(c^2 - 1)\hat{U}_{zz}  + \sin(\hat{U}) = 0, \quad c \neq 1. \label{eq:standwavesg}
	\end{align}
	Integrating \cref{eq:standwavesg} with respect to $ z $ yields:
	\begin{align*}
	\frac{1}{2}(c^{2} - 1)\hat{U}_{z}^{2} + 1 - \cos(\hat{U}) = E,
	\end{align*}
	where $ E $ is a constant of integration which we interpret as the total energy of the system. We follow the results established in \cite{Jones13} and \cite{Marangell2015}, where $ \hat{U}(z) $ is assumed to be periodic modulo $ 2\pi $. We denote the fundamental period of $ \hat{U} $ as $ T $, so that $ \hat{U}(z + T) = \hat{U}(z) \; (\mathrm{mod}\; 2\pi) $. Linearising \cref{eq:travellingsg} about this periodic standing wave solution, we write $ \hat{u} = \hat{U} + \epsilon p(z)e^{\lambda \tau} $ with $ \epsilon \ll 1 $ and equate $ O(\epsilon) $ terms, which yields the spectral problem:
	\begin{align}
	(c^2-1)p'' - 2c \lambda p' + \left(\lambda^2 + \cos(\hat{U})\right)p = 0. \label{eq:firstlsg}
	\end{align}
	The (Floquet) spectrum $ \sigma $ consists of all $ \lambda $ such that $ p \map{\R}{\C} $ is bounded. In \cite[Proposition 3.9]{Jones14}, Jones et. al. prove that $ \sigma $ has Hamiltonian symmetry $ \sigma = \sigma^{*} = -\sigma = -\sigma^{*} $. Consequently, any $ \lambda \in \sigma $ with non-zero real part implies an unstable eigenvalue. We make the substitution $ \lambda = i\zeta $ and henceforth use the phrase \textit{linearised sine-Gordon equation} to mean:
	\begin{align}
	p'' - \frac{2ic \zeta}{c^2-1} p' + \left(-\frac{\zeta^2}{c^2-1} + \frac{\cos(\hat{U})}{c^2-1}\right)p = 0. \label{eq:lsg}
	\end{align}
	An equivalent condition for $ \zeta \in \sigma $ is the existence of a non-trivial solution $ p(z) $ which can be written in Bloch form \`a la \cite{Marangell2015}:
	\begin{align*}
	p(z) = e^{-i\theta z/T}P(z),
	\end{align*} 
	where $ \theta \in \R $ is called the \textit{Floquet exponent}, and $ P(z) = P(z + T) $.	The solutions $ p(z) $ are quasi-periodic, since:
	\begin{align}
	p(z) = e^{i\theta}p(z+T). \label{eq:quasiperiodic}
	\end{align}
	We now reframe the spectral problem in \cref{eq:lsg} using the formalism of operator pencils, in particular drawing on the work of Markus \cite{Markus1988} and Kato \cite{Kato}. Given linear operators $ L_{0}, L_{1}, \dots, L_{n} $ with $L_{i} \map{X}{Y} $ for Banach Spaces $ X $ and $ Y $, then
	\begin{align*}
	\mathcal{L}(\lambda) \coloneqq  L_{n}\lambda^{n} +  L_{n-1}\lambda^{n-1} + \dots + L_{0}
	\end{align*}
	defines a polynomial operator pencil of degree $ n $ depending on the complex variable $ \lambda $ in an open set $ \lambda \in S \subset \mathbb{C}$ \cite[\S12.1]{Markus1988}. \Cref{eq:lsg} can be rewritten in the form $ \mathcal{L}_{SG}(\zeta)p = 0 $ where
	\begin{align}
	\mathcal{L}_{SG}(\zeta) \coloneqq -\frac{\zeta^2}{c^2-1}\mathbb{I} - \frac{2ic\zeta}{c^2-1}\partial_{z}+\left (\partial_{z}^2 + \frac{\cos(\hat{U})}{c^2-1}\right ) \label{eq:lsgpencil}
	\end{align}
	is a quadratic operator pencil. We narrow our focus to operator pencils which are holomorphic families of type (A). Such pencils $ \mathcal{L}(\lambda) $ have a domain $ B $ that is independent of the spectral variable $ \lambda $, and for each $ u \in B $, $ \mathcal{L}(\lambda)u $ is a holomorphic function of $ \lambda $. Kollar and Miller note in \cite{KM} that $ \mathcal{L}_{SG} $ defined in \cref{eq:lsgpencil}, acting on the spaces $ X = Y = L^{2}([0,T]) $ with domain $ D_{0} = C^{2}([0,T]) $, is a holomorphic family of type (A) with compact resolvent. Moreover, when $ \zeta \in \R $, $ \mathcal{L}_{SG}(\zeta) $ is self-adjoint. These qualities are necessary for Krein signatures to be well-defined \cite[Theorem 3.3]{KM}.\par
	As in \cite{KM}, for $ \mathcal{L}(\lambda) $ a holomorphic family of type (A), we say that $ \lambda_{0} \in \C $ is a \textit{characteristic value} if there exists a \textit{characteristic vector} $ u \neq 0 $ such that $ \mathcal{L}(\lambda_{0})u = 0$. The \textit{geometric multiplicity} of $ \lambda_{0} $ is $ \mathrm{dim}(\mathrm{ker}(\mathcal{L}(\lambda_{0}))) $. If we also assume that $ \mathcal{L}(\lambda) $ is self-adjoint and has compact resolvent, then the eigenvalue problem
	\begin{align*}
	\mathcal{L}(\lambda)u(\lambda) = \mu(\lambda)u(\lambda)
	\end{align*}
	can be solved for analytic functions $ u(\lambda), \mu(\lambda) $ at $ \lambda = \lambda_{0} $. In particular, if $ \text{dim}(\text{ker}(\mathcal{L}(\lambda_{0}))) = k$, then there exist exactly $ k $ analytic functions $ \mu_{1}(\lambda), \mu_{2}(\lambda), \dots, \mu_{k}(\lambda) $ called \textit{eigenvalue branches}, which vanish at $ \lambda = \lambda_{0} $. If we let $ m_{i} $ be the order of vanishing of the eigenvalue branch $ \mu_{i}(\lambda) $ at $ \lambda_{0} $ for $ 1 \leq i \leq k $, that is:
	\begin{align*}
	\mu_{i}(\lambda_{0}) = \mu_{i}'(\lambda_{0}) = \dots = \mu_{i}^{(m_{i}-1)}(\lambda_{0}) = 0, \quad \mu_{i}^{(m_{i})}(\lambda_{0}) \neq 0,
	\end{align*}
	then the algebraic multiplicity $ M $ of a characteristic value $ \lambda_{0} \in \R $ is given by:
	\begin{align*}
	M = \sum_{i = 1}^{k}m_{i}.
	\end{align*}
	Our main objective in this paper is to construct a new Evans function $ D(\zeta;\theta) $, $ D\map{\C\times [0,2\pi)}{\C} $ whose roots coincide exactly with the isolated characteristic values $ \zeta_{0} $ of $ \mathcal{L}_{SG}(\zeta) $ for a given $ \theta $, with $ D(\zeta_{0};\theta) $ vanishing to the order of algebraic multiplicity of $ \zeta_{0} $. Our secondary objective is to use this Evans function to calculate the Krein signatures of simple characteristic values of $ \mathcal{L}_{SG}(\zeta) $. A characteristic value is called \textit{simple} when its algebraic and geometric multiplicities are both 1. Kollar and Miller provide a full treatment of Krein signature theory in \cite{KM}, however the procedure for calculating the Krein signature for simple, isolated characteristic values is straightforward given the pencils in this paper. In particular, for an isolated characteristic value $ \lambda_{0} $ and its single eigenvalue branch $ \mu(\lambda) $ which vanishes to order 1 at $ \lambda = \lambda_{0} $, the graphical Krein signature can be calculated from \cite[Definition 3.5]{KM}:
	\begin{align}
	\kappa(\lambda_{0}) = \mathrm{sign}\left(\frac{d}{d\lambda}\mu(\lambda_{0})\right). \label{eq:KreinSig}
	\end{align}
	\section{Spectrum of the linearised sine-Gordon equation} \label{sec:s2}
	We begin by surveying the results of \cite{Jones13} and \cite{Marangell2015}. For each $ \zeta \in \sigma $, we define the \textit{principal fundamental solution matrix} as: 
	\begin{align}
	\mathbb{F}(z;\zeta) \coloneqq \begin{pmatrix}
	p_{1}(z;\zeta) & p_{2}(z;\zeta)\\
	p_{1}'(z;\zeta) & p_{2}'(z;\zeta)
	\end{pmatrix}, \label{eq:pfsmLSG}
	\end{align}
	where $ p_{1}(z;\zeta) $ and $ p_{2}(z;\zeta) $ are the unique solutions of \cref{eq:lsg} satisfying the initial conditions:
	\begin{align}
	\mathbb{F}(0;\zeta) = \begin{pmatrix}
	p_{1}(0;\zeta) & p_{2}(0;\zeta)\\
	p_{1}'(0;\zeta) & p_{2}'(0;\zeta)
	\end{pmatrix} = \begin{pmatrix}
	1 & 0\\
	0 & 1
	\end{pmatrix}. \label{eq:initialconditions}
	\end{align}
	We may write any solution $ p(z;\zeta) $ as a superposition of the fundamental solutions:
	\begin{align*}
	\begin{pmatrix}
	p(z;\zeta)\\
	p'(z;\zeta)
	\end{pmatrix} = \mathbb{F}(z;\zeta)\begin{pmatrix}
	C_{1}\\
	C_{2}
	\end{pmatrix}, \; C_{1}, C_{2} \in \C.
	\end{align*}
	Now if $ \zeta \in \sigma $, we can use \cref{eq:quasiperiodic} with $ z = 0 $:
	\begin{align*}
	\mathbb{F}(0;\zeta)\begin{pmatrix}
	C_{1}\\
	C_{2}
	\end{pmatrix} = e^{i\theta}\mathbb{F}(T;\zeta)\begin{pmatrix}
	C_{1}\\
	C_{2}
	\end{pmatrix},
	\end{align*}
	from which we have:
	\begin{align}
	\mathbb{F}(T;\zeta)\begin{pmatrix}
	C_{1}\\
	C_{2}
	\end{pmatrix} = e^{-i\theta}\begin{pmatrix}
	C_{1}\\
	C_{2}
	\end{pmatrix}. \label{eq:evaleq}
	\end{align}
	The matrix $ \mathbb{F}(T;\zeta) $ is called the \textit{monodromy matrix}, and its two eigenvalues $ \rho_{1}, \rho_{2} $ are referred to as \textit{Floquet multipliers} \cite{Jones14}. From \cref{eq:evaleq}, we identify $ \rho_{1} = e^{-i\theta} $, and we apply Abel's identity and the initial conditions in \cref{eq:initialconditions} to \cref{eq:lsg} to find that:
	\begin{align*}
	\det\left(\mathbb{F}(z;\zeta)\right) = \exp\left(\frac{2ic\zeta}{c^{2}-1}z\right).
	\end{align*}
	When $ z = T $, we have:
	\begin{align*}
	\rho_{1}\rho_{2} = \det\left(\mathbb{F}(T;\zeta)\right) = \exp\left(\frac{2icT\zeta}{c^{2}-1}\right),
	\end{align*}
	from which we conclude that:
	\begin{align}
	\rho_{2} = \exp\left(\frac{2icT\zeta}{c^{2}-1} + i\theta\right) . \label{eq:rho2}
	\end{align}
	We seek to make use of the results in \cite{Jones13} which connect the Floquet multipliers $ \rho_{1},\rho_{2} $ of \cref{eq:lsg} to \textit{Hill's equation} in \cref{eq:Hill}. Making the exponential transform \cite{Jones13}:
	\begin{align}
	q(z) = p(z)\exp\left(\frac{-ic\zeta}{c^{2} - 1}z\right), \label{eq:expTransform}
	\end{align}
	we transform \cref{eq:lsg} into the following form of Hill's equation:
	\begin{align}
	q'' + \left(\frac{\zeta^{2}}{(c^2-1)^{2}} + \frac{\cos(\hat{U})}{c^2-1}\right)q = 0 \label{eq:Hill}.
	\end{align}
	We define the principal fundamental solution matrix for \cref{eq:Hill} as:
	\begin{align}
	\mathbb{H}(z;\zeta) \coloneqq \begin{pmatrix}
	q_{1}(z;\zeta) & q_{2}(z;\zeta)\\
	q_{1}'(z;\zeta) & q_{2}'(z;\zeta)
	\end{pmatrix}, \quad \mathbb{H}(0;\zeta) = \begin{pmatrix}
	1 & 0\\
	0 & 1
	\end{pmatrix}. \label{eq:HillMonodromy}
	\end{align}
	We write $ \mathbb{H}(T;\zeta) $ for the monodromy matrix of \cref{eq:Hill}, and its two Floquet multipliers are denoted by $ \eta_{1}, \eta_{2} $. In \cite[Lemma 3.1]{Jones13}, Jones et al. prove that $ \rho_{1}, \rho_{2} $ are the Floquet multipliers of \cref{eq:lsg} if and only if 
	\begin{align}
	\eta_{1} = \exp{\frac{-ic\zeta T}{c^{2}-1}}\rho_{1},\; \eta_{2} = \exp{\frac{-ic\zeta T}{c^{2}-1}}\rho_{2} \label{eq:floquetmults}
	\end{align}
	are the Floquet multipliers of \cref{eq:Hill}. Consequently, the function:
	\begin{align}
	D_{2}(\zeta;\theta) \coloneqq \mathrm{tr}(\mathbb{H}(T;\zeta)) - (\eta_{1} + \eta_{2}) \label{eq:d2}
	\end{align}
	vanishes, at least to first order, at precisely the values $ \zeta = \zeta_{0} $ which are characteristic values of the linearised sine-Gordon equation for a given Floquet exponent $ \theta $. We can calculate $ \eta_{1} $ and $ \eta_{2} $ directly using \cref{eq:rho2,,eq:floquetmults}:
	\begin{align*}
	\eta_{1} = \exp\left(-\frac{ic\zeta T}{c^{2}-1} - i\theta\right), \; \eta_{2} = \exp\left(\frac{ic\zeta T}{c^{2}-1} + i\theta\right),
	\end{align*}
	which indeed obeys the condition that $ \eta_{1}\eta_{2} = \det(\mathbb{H}(z;\zeta)) = 1 $ by Abel's identity. \Cref{eq:d2} then simplifies to:
	\begin{align}
	D_{2}(\zeta;\theta) \coloneqq \mathrm{tr}(\mathbb{H}(T;\zeta)) - 2\cos\left(\frac{c\zeta T}{c^{2}-1} + \theta\right) \label{eq:newEvans}
	\end{align}
	We pause here to include a relevant result from \cite[Definition 3.5]{Jones14}. The authors provide an Evans function for quasi-periodic solutions $ p(z) = e^{i\theta}p(z+T) $ to the linearised sine-Gordon \cref{eq:lsg}:
	\begin{align}
	D_{1}(\zeta;\theta) \coloneqq \det(\mathbb{F}(T;\zeta) - e^{-i\theta}\mathbb{I}) \label{eq:D1}.
	\end{align}
	\begin{Theorem}
		The function
		\begin{align*}
		D_{2}(\zeta;\theta) \coloneqq \mathrm{tr}(\mathbb{H}(T;\zeta)) - 2\cos\left(\frac{c\zeta T}{c^{2}-1} + \theta\right) 
		\end{align*}
		is an Evans function for characteristic values of $ \mathcal{L}_{SG}(\zeta) $ defined in \cref{eq:lsgpencil} parametrised by the Floquet exponent $ \theta $.
	\end{Theorem}
	\begin{proof}
		Magnus and Winkler proved that $ \mathrm{tr}(\mathbb{H}(T;\zeta)) $ is an entire function of $ \zeta $ \cite[Theorem 2.2]{MW}, so $ D_{2}(\zeta;\theta) $ is itself an entire function of $ \zeta $. We now prove that $ D_{1}(\zeta;\theta) $ (defined in \cref{eq:D1}) and $ D_{2}(\zeta;\theta) $ vanish at the same values $ \zeta_{0} $ to precisely the same degree, thus qualifying $ D_{2}(\zeta;\theta) $ as an Evans function. We proceed by rewriting $ D_{1}(\zeta;\theta) $ and $ D_{2}(\zeta;\theta) $ in terms of fundamental solutions in \cref{eq:pfsmLSG} and \cref{eq:HillMonodromy}:
		\begin{align}
		D_{1}(\zeta;\theta) &= \det\begin{pmatrix}p_{1}(T;\zeta) - e^{-i\theta} & p_{2}(T;\zeta)\\
		p_{1}'(T;\zeta) & p_{2}'(T;\zeta) - e^{-i\theta}\\
		\end{pmatrix},\\
		&= \exp\left(\frac{2icT\zeta}{c^{2}-1}\right) - e^{-i\theta}(p_{1}(T;\zeta) + p_{2}'(T;\zeta)) + e^{-2i\theta}. \label{eq:D1Evans}
		\end{align}
		Similarly:
		\begin{align*}
		D_{2}(\zeta;\theta) = q_{1}(T;\zeta) + q_{2}'(T;\zeta) - 2\cos\left(\frac{c\zeta T}{c^{2}-1} + \theta\right).
		\end{align*}
		We define $ r_{1}(z;\zeta) $ and $ r_{2}(z;\zeta) $ by using the exponential transform in \cref{eq:expTransform} applied to $ p_{1} $ and $ p_{2} $:
		\begin{align*}
		r_{1}(z;\zeta) &\coloneqq p_{1}(z;\zeta)\exp\left(\frac{-ic\zeta}{c^{2}-1}z\right),\\
		r_{2}(z;\zeta) &\coloneqq p_{2}(z;\zeta)\exp\left(\frac{-ic\zeta}{c^{2}-1}z\right),
		\end{align*}
		and we note that $ r_{1}, r_{2} $ are solutions of \cref{eq:Hill}. We use the initial conditions of $ p_{1}, p_{2} $ to express $ r_{1}, r_{2} $ in the basis of fundamental solutions $ q_{1}, q_{2} $, which yields:
		\begin{align*}
		r_{1}(z;\zeta) &= q_{1}(z;\zeta) - \frac{ic\zeta}{c^{2}-1}q_{2}(z;\zeta)\\
		r_{2}(z;\zeta) &= q_{2}(z;\zeta).
		\end{align*}
		We can now write $ p_{1}, p_{2} $ in terms of $ q_{1}, q_{2} $:
		\begin{align*}
		p_{1}(z;\zeta) &= \exp\left(\frac{ic\zeta}{c^{2}-1}z\right)\left(q_{1}(z;\zeta) - \frac{ic\zeta}{c^{2}-1}q_{2}(z;\zeta)\right)\\
		p_{2}(z;\zeta) &= \exp\left(\frac{ic\zeta}{c^{2}-1}z\right)q_{2}(z;\zeta).
		\end{align*}
		Finally, we substitute these expressions into \cref{eq:D1Evans}, which yields:
		\begin{align}
		D_{1}(\zeta;\theta) = -\exp\left(\frac{ic\zeta T}{c^{2}-1} - i\theta\right)D_{2}(\zeta;\theta). \label{eq:EvansFunctionsEquation}
		\end{align}
		It follows that $ D_{1}(\zeta;\theta) = 0 \iff D_{2}(\zeta;\theta) = 0 $. Suppose that $ \zeta_{0} $ is a zero of $ D_{1}(\zeta;\theta) $ with degree of vanishing $ n $:
		\begin{align*}
		D_{1}(\zeta_{0};\theta) = D_{1}'(\zeta_{0};\theta) = \dots = D_{1}^{(n-1)}(\zeta_{0};\theta) = 0, \quad D_{1}^{(n)}(\zeta_{0};\theta) \neq 0.
		\end{align*}
		For some $ j \leq n $, with base case $ j = 1 $, we assume inductively that:
		\begin{align}
		D_{2}(\zeta_{0};\theta) = D_{2}'(\zeta_{0};\theta) = \dots = D_{2}^{(j-1)}(\zeta_{0};\theta) = 0. \label{eq:inductivehypothesis}
		\end{align}
		We apply the product rule to \cref{eq:EvansFunctionsEquation} $ j  $ times, which yields:
		\begin{align*}
		D_{1}^{(j)}(\zeta_{0};\theta) = -\exp\left (\frac{icT\zeta_{0}}{c^2-1}-i\theta\right )\sum_{k = 0}^{j}{j \choose k}\left (\frac{icT}{c^2-1}\right )^{j-k}D_{2}^{(k)}(\zeta_{0};\theta).
		\end{align*}
		Using the inductive hypothesis in \cref{eq:inductivehypothesis}, we have:
		\begin{align*}
		D_{1}^{(j)}(\zeta_{0};\theta) = -\exp\left (\frac{icT\zeta_{0}}{c^2-1}-i\theta\right )D_{2}^{(j)}(\zeta_{0};\theta),
		\end{align*}
		and hence $ D_{1}^{(j)}(\zeta_{0};\theta) = 0 \implies D_{2}^{(j)}(\zeta_{0};\theta) = 0 $. By induction, this is true for $ j = 1, \dots, n - 1 $, and we check that:
		\begin{align*}
		D_{1}^{(n)}(\zeta_{0};\theta) = -\exp\left (\frac{icT\zeta_{0}}{c^2-1}-i\theta\right )D_{2}^{(n)}(\zeta_{0};\theta),
		\end{align*}
		meaning that $ D_{1}^{(n)}(\zeta_{0};\theta) \neq 0 \implies D_{2}^{(n)}(\zeta_{0};\theta) \neq 0 $. Given $ \zeta_{0} $ a zero of order $ n $ for $ D_{1}(\zeta;\theta) $, then $ \zeta_{0} $ is also a zero of order $ n $ for $ D_{2}(\zeta;\theta) $. The zeros of $ D_{1}(\zeta;\theta) $ are the only zeros of $ D_{2}(\zeta;\theta) $, which is proved by rearranging \cref{eq:EvansFunctionsEquation} to:
		\begin{align*}
		D_{2}(\zeta;\theta) = - \exp\left(-\frac{ic\zeta T}{c^{2}-1} + i\theta\right)D_{1}(\zeta;\theta)
		\end{align*}
		and following the same proof by induction as above. Thus $ D_{2}(\zeta;\theta) $ is an Evans function for quasi-periodic solutions of the linearised sine-Gordon equation.
	\end{proof}
	\begin{corollary}
		The function
		\begin{align*}
		D_{3}(\zeta;\theta) = \det\left(\mathbb{H}(T;\zeta) - \exp\left(-\frac{ic\zeta T}{c^{2}-1} - i\theta\right)\mathbb{I}\right)
		\end{align*}
		is also an Evans function for quasi-periodic solutions of the linearised sine-Gordon equation.
	\end{corollary}
	\begin{proof}
		Upon expansion we have:
		\begin{align*}
		D_{3}(\zeta;\theta) &= \det(\mathbb{H}(T;\zeta)) - \exp\left(-\frac{ic\zeta T}{c^{2}-1} - i\theta\right)\mathrm{tr}(\mathbb{H}(T;\zeta)) + \exp\left(-\frac{2ic\zeta T}{c^{2}-1} - 2i\theta\right)\\
		&= - \exp\left(-\frac{ic\zeta T}{c^{2}-1} - i\theta\right)\left(\mathrm{tr}(\mathbb{H}(T;\zeta)) - \exp\left(\frac{ic\zeta T}{c^{2}-1} + i\theta\right) - \exp\left(-\frac{ic\zeta T}{c^{2}-1} - i\theta\right)\right)\\
		&= -\exp\left(-\frac{ic\zeta T}{c^{2}-1} - i\theta\right)D_{2}(\zeta;\theta).
		\end{align*}
	\end{proof}
	We can use these new Evans functions to compute the spectrum of the linearised sine-Gordon equation using only the solutions of Hill's \cref{eq:Hill}. In \cref{fig:f1a,fig:f1b,fig:f1d} we reproduce the results of \cite{Jones13}, while \cref{fig:f1c} is a reproduction of a result in \cite{Marangell2015}. These diagrams are produced by substituting complex values of $ \zeta $ into the Evans function $ D_{2}(\zeta;\theta) $ and applying a root-finder to find the values of $ \theta $ at which $ D_{2}(\zeta;\theta) $ vanishes. Our diagrams are the same as those in the referenced papers.\par
	\begin{figure}[t]
		\centering
		\begin{subfigure}[t]{0.4 \textwidth}
			\captionsetup{justification=centering,margin=1cm}
			\centering
			\includegraphics[width=\linewidth]{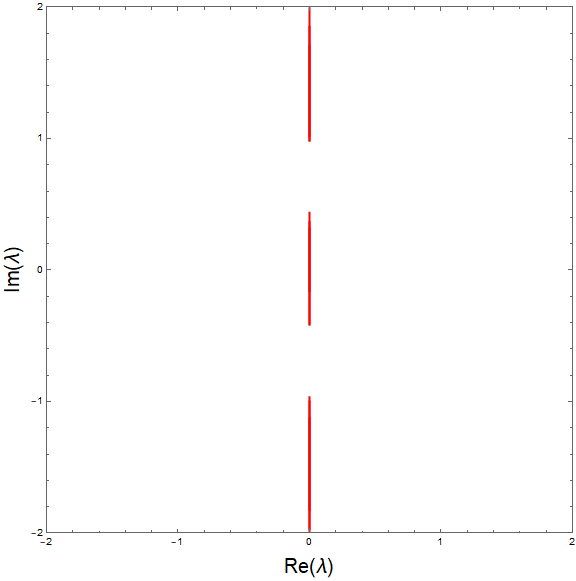}
			\caption{Subluminal rotational. $E = -0.5,\; c = 0.5 $.}
			\label{fig:f1a}
		\end{subfigure}
		\begin{subfigure}[t]{0.4 \textwidth}
			\captionsetup{justification=centering,margin=1cm}
			\centering
			\includegraphics[width=\linewidth]{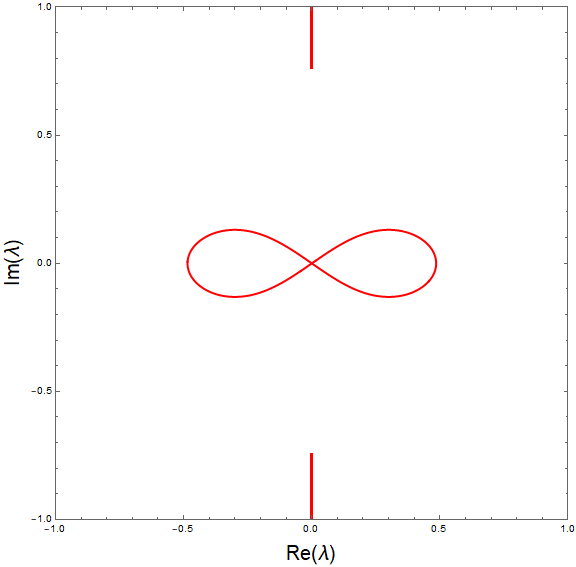}
			\caption{Subluminal librational. $E = 0.5,\; c = 0.5 $.}
			\label{fig:f1b}
		\end{subfigure}
		\begin{subfigure}[t]{0.4 \textwidth}
			\captionsetup{justification=centering,margin=1cm}
			\centering
			\includegraphics[width=\linewidth]{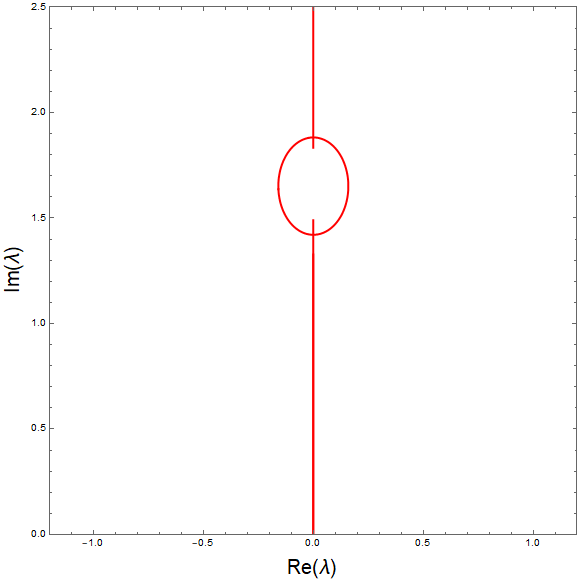}
			\caption{Upper halfplane. Superluminal rotational. $E = 6,\; c = 1.45 $.}
			\label{fig:f1c}
		\end{subfigure}
		\begin{subfigure}[t]{0.4 \textwidth}
			\captionsetup{justification=centering,margin=1cm}
			\centering
			\includegraphics[width=\linewidth]{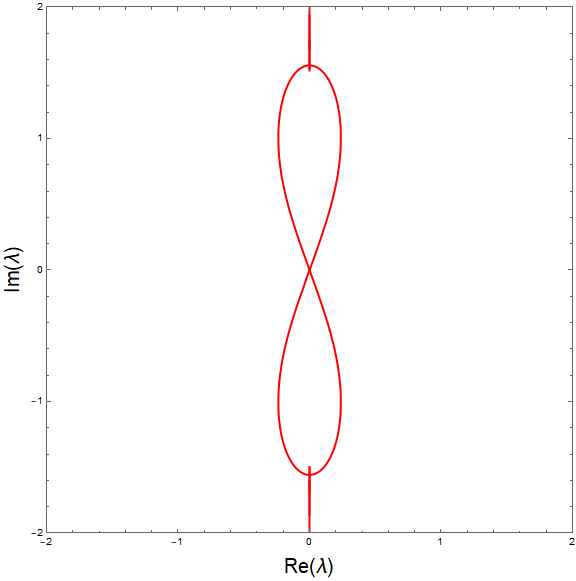}
			\caption{Superluminal librational. $E = 1.5,\; c = 2 $.}
			\label{fig:f1d}
		\end{subfigure}
		\caption{Numerical plots of the spectra $ \sigma $ of various periodic travelling wave solutions to the sine-Gordon \cref{eq:sG}, calculated by finding solutions to the Evans function $ D_{2}(\zeta;\theta) $  with $ \zeta \in \C $ and for all $ \theta \in [0, 2\pi) $. Recall that $ \lambda = i\zeta $. }
		\label{fig:f1}
	\end{figure}
	\subsection{Krein signatures of the linearised sine-Gordon equation}
	Making the restriction that $ \zeta \in \R $, then $ \mathcal{L}_{SG}(\zeta) $ is a self-adjoint, holomorphic family of type (A) with compact resolvent \cite[Example 11]{KM}, meaning that its characteristic values have well-defined Krein signatures. The Evans function $ D_{2}(\zeta;\theta) $ can be turned into a so-called \textit{Evans-Krein function}, used in calculating the Krein signatures of isolated characteristic values $ \zeta_{0} $. Following Kollar and Miller in \cite[Theorem 4.2]{KM}, an Evans-Krein function $ E(\lambda;\mu) $ for an operator pencil $ \mathcal{L}(\lambda) $ is an Evans function for the $ \mu $-parametrised pencil:
	\begin{align*}
	\mathcal{K}(\lambda;\mu) \coloneqq \mathcal{L}(\lambda) - \mu\mathbb{I}.
	\end{align*}
	In \cite[\S4.4]{KM}, Kollar and Miller prove for an isolated, simple characteristic value $ \lambda_{0} $ that:
	\begin{align}
	\mu'(\lambda_{0}) = -\frac{E_{\lambda}(\lambda_{0};0)}{E_{\mu}(\lambda_{0};0)}. \label{eq:muderivative}
	\end{align}
	\Cref{eq:muderivative} allows us to calculate the Krein signature via \cref{eq:KreinSig}. In our case, we consider the related pencil
	\begin{align*}
	\mathcal{K}_{SG}(\zeta) \coloneqq \mathcal{L}_{SG}(\zeta) - \mu\mathbb{I},
	\end{align*}
	which has Evans-Krein function:
	\begin{align}
	E_{SG}(\zeta;\mu) = \mathrm{tr}(\mathbb{H}(T;\zeta, \mu)) - 2\cos\left(\frac{c\zeta T}{c^{2}-1} + \theta\right). \label{eq:lsgEvansKrein}
	\end{align}
	The monodromy matrix $ \mathbb{H}(T;\zeta,\mu) $ is defined as in \cref{eq:HillMonodromy}, however the related Hill equation is:
	\begin{align}
	\left(\frac{\zeta^{2}}{(c^{2}-1)^{2}} - \mu\right)q + \left(\partial_{z}^{2} + \frac{\cos(\hat{U})}{c^{2} - 1}\right)q = 0. \label{eq:HillRelated}
	\end{align}
	We can use the substitution $ v(\zeta, \mu) = \frac{\zeta^{2}}{(c^{2} - 1)^{2}} - \mu $ to calculate the partial derivatives of $ E_{SG} $:
	\begin{align*}
	\frac{\partial}{\partial \zeta}E_{SG}(\zeta;\mu) &= \frac{\partial v}{\partial \zeta}\frac{\partial}{\partial v}\mathrm{tr}(\mathbb{H}(T;\zeta, \mu)) + \frac{2cT}{c^{2} - 1}\sin\left(\frac{c\zeta T}{c^{2}-1} + \theta\right)\\
	&= \frac{2\zeta}{(c^{2} - 1)^{2}}\frac{\partial}{\partial v}\mathrm{tr}(\mathbb{H}(T;\zeta, \mu)) + \frac{2cT}{c^{2} - 1}\sin\left(\frac{c\zeta T}{c^{2}-1} + \theta\right)\\
	\frac{\partial}{\partial \mu}E_{SG}(\zeta;\mu) &= \frac{\partial v}{\partial \mu}\frac{\partial}{\partial v}\mathrm{tr}(\mathbb{H}(T;\zeta, \mu))\\
	&= -\frac{\partial}{\partial v}\mathrm{tr}(\mathbb{H}(T;\zeta, \mu)).
	\end{align*}
	Using these results with \cref{eq:muderivative} we have:
	\begin{align}
	\mu'(\zeta_{0}) &= \frac{2\zeta_{0}}{(c^{2} - 1)^{2}} +  \frac{\frac{2cT}{c^{2} - 1}\sin\left(\frac{c\zeta_{0} T}{c^{2}-1} + \theta\right)}{\frac{\partial}{\partial v}\mathrm{tr}(\mathbb{H}(T;\zeta_{0}, 0))}. \label{eq:muVderiv}
	\end{align}
	We have from \cite[Corollary 2.1, Theorem 2.2]{MW} that
	\begin{align*}
	D_{Hill}(\zeta) \coloneqq \mathrm{tr}(\mathbb{H}(T;\zeta_{0})) - 2
	\end{align*}
	is an Evans function for characteristic values of periodic solutions to Hill's \cref{eq:Hill}. So we have:
	\begin{align*}
	D_{Hill}'(\zeta_{0}) &= \frac{\partial v}{\partial \zeta}\frac{\partial}{\partial v}\mathrm{tr}(\mathbb{H}(T;\zeta_{0}, 0))\\
	\implies \frac{\partial}{\partial v}\mathrm{tr}(\mathbb{H}(T;\zeta_{0}, 0)) &= \frac{(c^{2} - 1)^{2}}{2\zeta_{0}}D_{Hill}'(\zeta_{0}).
	\end{align*}
	We make this substitution in \cref{eq:muVderiv} because it is easier numerically to compute $ D_{Hill}'(\zeta_{0}) $. Hence we have:
	\begin{align}
	\mu'(\zeta_{0}) = \frac{2\zeta_{0}}{(c^{2} - 1)^{2}} +\frac{4c\zeta_{0}T}{(c^{2} - 1)^{3}} \frac{\sin\left(\frac{c\zeta_{0} T}{c^{2}-1} + \theta\right)}{D_{Hill}'(\zeta_{0})}. \label{eq:mudash}
	\end{align}
	Finally, using \cref{eq:KreinSig}, we compute the Krein signature of simple characteristic values for quasi-periodic solutions of the linearised sine-Gordon \cref{eq:lsg}:
	\begin{align}
	\kappa(\zeta_{0}) = \mathrm{sign}\left(\frac{2\zeta_{0}}{(c^{2} - 1)^{2}} +\frac{4c\zeta_{0}T}{(c^{2} - 1)^{3}} \frac{\sin\left(\frac{c\zeta_{0} T}{c^{2}-1} + \theta\right)}{D_{Hill}'(\zeta_{0})}\right). \label{eq:kreinsg}
	\end{align}
	We pause to note that the situations where $ \zeta_{0} = 0 $ or $ D_{Hill}'(\zeta_{0}) = 0 $ pose problems with the above derivation. Firstly, $ \zeta_{0} = 0 $ is never a simple characteristic value of the linearised sine-Gordon equation. This is a natural consequence of the system being a second degree autonomous Hamiltonian. For the purposes of completeness, we sketch a proof of this fact. We observe that:
	\begin{align*}
	\mathcal{L}_{SG}(0)\hat{U}'(z) &= \hat{U}_{zzz} + \frac{\cos(\hat{U})}{c^{2}-1}\hat{U}_{z}\\
	&= 0,
	\end{align*}
	where the last step follows by differentiating \cref{eq:standwavesg}:
	\begin{align*}
	(c^2 - 1)\hat{U}_{zz}  + \sin(\hat{U}) = 0.
	\end{align*} 
	The function $ \hat{U}'(z) $ is $ T $-periodic, and so it has Floquet exponent $ \theta = 0 $. Noting that $ \mathrm{tr}(\mathbb{H}(T;\zeta)) $ is analytic (by \cite[Theorem 2.2]{MW}) and even in $ \zeta $ (since Hill's \cref{eq:Hill} is only dependent on $ \zeta^{2} $), then 
	\begin{align*}
	D_{2}(\zeta;0)=\mathrm{tr}(\mathbb{H}(T;\zeta)) - 2\cos\left(\frac{c\zeta T}{c^{2}-1}\right)
	\end{align*}
	is analytic and even in $ \zeta $. Hence $ D_{2}'(0;0) = 0 $, and so the multiplicity of $ \zeta_{0} = 0 $ is at least 2. Moreover, $ \theta = 0 $ is the unique Floquet exponent of $ \zeta_{0} = 0 $, since $ D_{2}(0;0) = 0 $ implies that $ \mathrm{tr}(\mathbb{H}(T;0)) = 2 $, and hence 
	\begin{align*}
	D_{2}(0;\theta) = 2 - 2\cos\left(\theta\right) = 0 \iff \theta = 0
	\end{align*}
	with $ \theta \in [0,2\pi) $. An alternative proof that $ \zeta_{0} = 0 $ is a characteristic value with even multiplicity for the more general linearised nonlinear Klein-Gordon equation is in \cite[Lemma 6.2]{Jones14}. As for the values $ x_{0} $ when 
	\begin{align*}
	D_{Hill}'(x_{0}) =  \frac{\partial}{\partial \zeta}\left(\mathrm{tr}(\mathbb{H}(T;\zeta))\right)\big \lvert_{\zeta = x_{0}} = 0,
	\end{align*}
	we refer to Magnus and Winkler's oscillation theorem \cite[Theorem 2.1]{MW}, which states that
	\begin{align*}
	D_{Hill}'(x_{0}) = 0 \iff \lvert \mathrm{tr}(\mathbb{H}(T;x_{0}))\rvert \geq 2.
	\end{align*}
	If $ \lvert \mathrm{tr}(\mathbb{H}(T;x_{0}))\rvert > 2 $, then $ \lvert D_{2}(x_{0};\theta)\rvert > 0 $, so $ x_{0} $ is not a characteristic value. If:
	\begin{align*}
	\mathrm{tr}(\mathbb{H}(T;x_{0})) = (-1)^{j}2
	\end{align*}
	for $ j = 0,1 $, then there exists
	\begin{align*}
	\theta_{0} = j\pi-\frac{cx_{0}T}{c^{2} - 1} \; (\mathrm{mod}\; 2\pi)
	\end{align*}
	such that $ D_{2}(x_{0};\theta_{0}) = D_{2}'(x_{0};\theta_{0}) = 0 $, meaning that $ x_{0} $ has multiplicty greater than 1. Concretely, our formula in \cref{eq:kreinsg} is able to calculate the Krein signature of all simple characteristic values of the linearised sine-Gordon \cref{eq:lsg}. These results are immediately generalisable to the Klein-Gordon case, where $ \cos(\hat{U}) $ is replaced with $ V''(\hat{U}) $, and we will use these results freely in \cref{sec:nlKG}.\par
	The advantage of using $ D_{2}(\zeta;\theta) $ over $ D_{1}(\zeta;\theta) $ when computing Krein signatures is that we did not have to explicitly compute a partial derivative of an Evans-Krein function with respect to $ \mu $. Calculating $ \mu $-derivatives must be done numerically and involves solving for a set of fundamental solutions at several values of $ \mu $ for each simple characteristic value $ \zeta_{0} $. However, our Evans function bypasses this computational overhead since we can make the subsitution $ v(\zeta,\mu) = \frac{\zeta^{2}}{(c^{2}-1)^{2}} - \mu $ in the Evans-Krein function, which takes advantage of the dependence of \cref{eq:HillRelated} on $ \frac{\zeta^{2}}{(c^{2}-1)^{2}} - \mu $. Such a substitution is not possible when computing Krein signatures using $ D_{1}(\zeta;\theta) $, since this function is written in terms of the fundamental solutions to the linearised sine-Gordon \cref{eq:lsg}. For the sake of exposition, we can adapt $ D_{1}(\zeta;\theta) $ into an Evans-Krein function $ E_{1}(\zeta;\mu,\theta) $:
	\begin{align*}
	E_{1}(\zeta;\mu,\theta) &= \det(\mathbb{F}(T;\zeta,\mu) - e^{-i\theta}\mathbb{I})\\
	&= \exp\left(\frac{2icT\zeta}{c^{2}-1}\right) - e^{-i\theta}(\mathrm{tr}(\mathbb{F}(T;\zeta,\mu))) + e^{-2i\theta},
	\end{align*}
	where we have expanded $ E_{1}(\zeta;\mu,\theta) $ as in \cref{eq:D1Evans}. The monodromy matrix $ \mathbb{F}(T;\zeta,\mu) $ is made up of fundamental solutions to the related linearised sine-Gordon equation:
	\begin{align}
	\left(-\frac{\zeta^2}{c^2-1} - \mu\right)p - \frac{2ic \zeta}{c^2-1}\partial_{z} p + \left(\partial_{z}^{2} + \frac{\cos(\hat{U})}{c^2-1}\right)p = 0. \label{eq:Relatedlsg}
	\end{align}
	Since a substition for $ \zeta $ and $ \mu $ is not possible in \cref{eq:Relatedlsg}, we rely on direct computation of the derivatives of $ E_{1}(\zeta;\mu,\theta) $:
	\begin{align*}
	\frac{\partial}{\partial \zeta}E_{1}(\zeta;\mu,\theta) &= \frac{2icT}{c^{2}-1}\exp\left(\frac{2icT\zeta}{c^{2}-1}\right) - e^{-i\theta}\frac{\partial}{\partial \zeta}\left(\mathrm{tr}\left(\mathbb{F}(T;\zeta,\mu)\right)\right)\\
	\frac{\partial}{\partial \mu}E_{1}(\zeta;\mu,\theta) &= - e^{-i\theta}\frac{\partial}{\partial \mu}\left(\mathrm{tr}\left(\mathbb{F}(T;\zeta,\mu)\right)\right).
	\end{align*}
	Using \cref{eq:muderivative}, we now have:
	\begin{align*}
	\mu'(\zeta_{0}) = -\frac{\frac{2icT}{c^{2}-1}\exp\left(\frac{2icT\zeta_{0}}{c^{2}-1} + i\theta\right)}{\frac{\partial}{\partial \mu}\left(\mathrm{tr}\left(\mathbb{F}(T;\zeta_{0},0)\right)\right)} + \frac{\frac{\partial}{\partial \zeta}\left(\mathrm{tr}\left(\mathbb{F}(T;\zeta_{0},0)\right)\right)}{\frac{\partial}{\partial \mu}\left(\mathrm{tr}\left(\mathbb{F}(T;\zeta_{0},0)\right)\right)}.
	\end{align*}
	Using $ E_{1}(\zeta;\mu,\theta) $ to compute $ \mu'(\zeta_{0}) $, we must numerically differentiate
	\begin{align*}
	\mathrm{tr}\left(\mathbb{F}(T;\zeta,\mu)\right) = p_{1}(T;\zeta,\mu) + p_{2}'(T;\zeta,\mu)
	\end{align*}
	at $ (\zeta,\mu) = (\zeta_{0},0) $ with respect to both $ \zeta $ and $ \mu $. For a first order approximation of the $\mu$-derivative, we need to solve \cref{eq:Relatedlsg} for some $ \mu \ll 1 $ at each characteristic value $ \zeta_{0} $ of interest. Higher order approximations will require solving \cref{eq:Relatedlsg} for several values of $ \mu $ at each $ \zeta_{0} $, however \cref{eq:mudash} shows that we can find $ \mu'(\zeta_{0}) $ in terms of only a $ \zeta $-derivative of $ D_{Hill}(\zeta) = q_{1}(T;\zeta) + q_{2}'(T;\zeta) - 2 $ at $ \zeta = \zeta_{0} $. Hence, using our new Evans function $ D_{2}(\zeta;\theta) $ rather than $ D_{1}(\zeta;\theta) $ results in a more elegant calculation of Krein signatures. \par
	\begin{figure}[t]
		\centering
		\begin{subfigure}[t]{0.49 \textwidth}
			\captionsetup{justification=centering,margin=1cm}
			\centering
			\includegraphics[width=\linewidth]{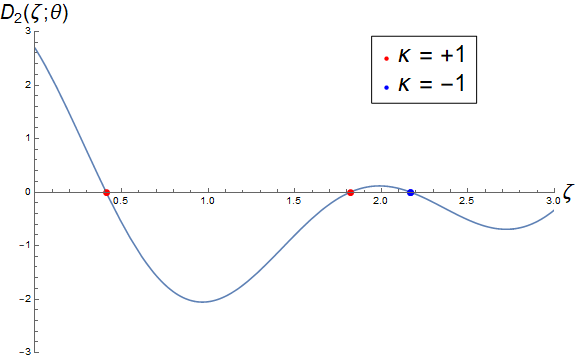}
			\caption{$ \theta = 4.36 $.}
			\label{fig:f2a}
		\end{subfigure}
		\begin{subfigure}[t]{0.49 \textwidth}
			\captionsetup{justification=centering,margin=1cm}
			\centering
			\includegraphics[width=\linewidth]{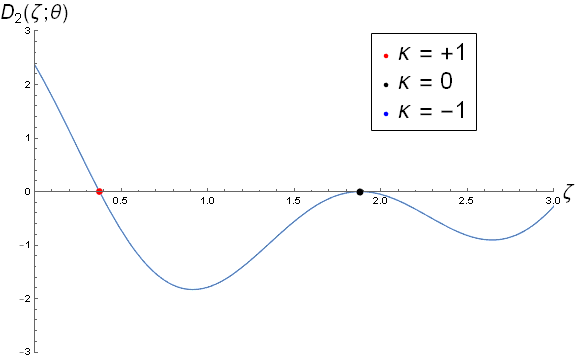}
			\caption{$ \theta = 4.525 $.}
			\label{fig:f2b}
		\end{subfigure}
		\begin{subfigure}[t]{0.49 \textwidth}
			\captionsetup{justification=centering,margin=1cm}
			\centering
			\includegraphics[width=\linewidth]{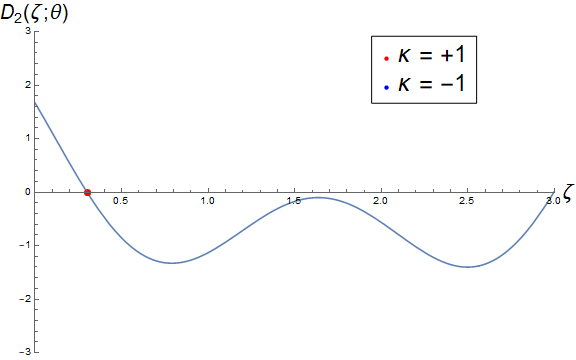}
			\caption{$ \theta = 4.88 $.}
			\label{fig:f2c}
		\end{subfigure}
		\begin{subfigure}[t]{0.49 \textwidth}
			\captionsetup{justification=centering,margin=1cm}
			\centering
			\includegraphics[width=\linewidth]{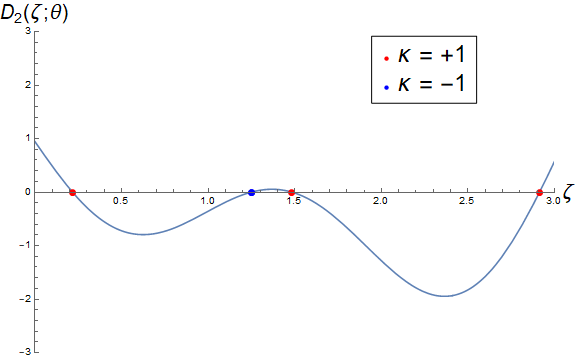}
			\caption{$ \theta = 5.27 $.}
			\label{fig:f2d}
		\end{subfigure}
		\caption{Numerical plots of $ D_{2}(\zeta;\theta) $ showing the Krein signatures of isolated characteristic values of the linearised sine-Gordon \cref{eq:lsg}, linearised around the superluminal rotational wave in \cref{fig:f1c} with $ E = 6, c = 1.45 $. A collision of opposite Krein signatures results in a Hamiltonian-Hopf bifurcation. Krein signatures with $ \kappa = 1 $ are shown online with red dots, while Krein signatures with $ \kappa = -1 $ are denoted by blue dots. \Cref{fig:f2b} shows the bifurcation point which corresponds to a Krein signature of $ \kappa = 0 $, denoted by a black dot.}
		\label{fig:f2}
	\end{figure}
	In \cref{fig:f2} we use \cref{eq:kreinsg} to numerically calculate the Krein signatures of isolated characteristic values in the case of the superluminal rotational wave in \cref{fig:f1c}. As the bifurcation parameter $ \theta $ is varied, we observe a collision between two characteristic values of opposite Krein signature, resulting in a Hamiltonian-Hopf bifurcation at $ \theta \approx 4.53 $ where these two characteristic values enter the complex plane. The two characteristic values then bifurcate back onto the imaginary axis at $ \theta \approx 5.16 $.
	\section{The nonlinear Klein-Gordon equation} \label{sec:nlKG}
	We now consider the more general nonlinear Klein-Gordon equation:
	\begin{align}
	u_{tt} - u_{xx} + V'(u) = 0, \label{eq:nlKG}
	\end{align}
	where $ u(x,t) \map{\R \times [0,+\infty)}{\R} $ and $ V(u)\map{\R}{\R} $ is a $ C^{2} $ potential. It is possible to recover the sine-Gordon \cref{eq:sG} by setting $ V(u) = 1 - \cos(u) $. Similar to our derivation of the linearised sine-Gordon equation, we have the linearised nonlinear Klein-Gordon equation:
	\begin{align}
	p'' - \frac{2ic \zeta}{c^2-1} p' + \left(-\frac{\zeta^2}{c^2-1} + \frac{V''(\hat{U})}{c^2-1}\right)p = 0, \label{eq:lnlkg}
	\end{align}
	where $ \hat{U}(z) $ satisfies
	\begin{align}
	(c^2 - 1)\hat{U}_{zz}  + V'(\hat{U}) = 0, \label{eq:standwavenlkg}
	\end{align}
	and has period $ T $. The case when $ V(u) $ is a periodic function of $ u $ has been studied extensively and we point the reader to \cite{Jones14} for a thorough analysis. Provided that $ \hat{U}(z) $ is periodic, then $ V(\hat{U}) $ is also periodic, making available the theory of \cite{Jones14} with the caveat that rotational waves will not be observed when $ V $ is not periodic. In fact, all the results of \cref{sec:s2} are immediately generalisable to any $ V \in C^{2} $, with the related Hill's \cref{eq:Hill} becoming:
	\begin{align}
	q'' + \left(\frac{\zeta^{2}}{(c^2-1)^{2}} + \frac{V''(\hat{U})}{c^2-1}\right)q = 0, \label{eq:nlkgHill}
	\end{align}
	and the Evans functions $ D_{2}(\zeta;\theta) $ and $ D_{3}(\zeta;\theta) $ remaining unchanged. The spectrum of the nonlinear Klein-Gordon equation exhibits Hamiltonian symmetry \cite[Proposition 3.9]{Jones14}. In particular, if $ p(z) $ satisfies \cref{eq:lnlkg} for $ \zeta \in \sigma $, then taking complex conjugates implies that $ p^{*}(z) $ satisfies \cref{eq:lnlkg} for $ \zeta^{*} $, and making the transformation 
	\begin{align*}
	p(z) = e^{\frac{-2ic\zeta}{c^{2} - 1}z}r(z)
	\end{align*}
	implies that $ r(z) $ satisfies \cref{eq:lnlkg} for $ -\zeta $.\par
	As a point of contrast to the sine-Gordon potential $ V(u) = 1 - \cos(u) $, we have chosen to consider the non-periodic potential $ V(u) = \frac{1}{4}u^{4} - \frac{1}{2}u^{2} $. The nonlinear Klein-Gordon equation with this non-periodic potential is known as the $ \phi^{4} $-model; \cite{Pal2020} includes an analysis of the spectral and orbital stability of subluminal periodic wavetrains. As before, we integrate \cref{eq:standwavenlkg} once which introduces the energy parameter $ E $:
	\begin{align*}
	\frac{1}{2}(c^{2} - 1)\hat{U}_{z}^{2} + \frac{1}{4}\hat{U}^{4} - \frac{1}{2}\hat{U}^{2} = E.
	\end{align*}
	\Cref{fig:f3} shows the phase portraits for subluminal and superluminal travelling wave solutions to \cref{eq:standwavenlkg}. In the subluminal case in \cref{fig:f3a}, we note that the separatrix corresponds to $ E = -\frac{1}{4} $, and we have that $ -\frac{1}{4} < E < 0 $. For superluminal waves in \cref{fig:f3b}, $ E > 0 $ corresponds to waves outside the homoclinic orbit, while $ -\frac{1}{4} < E < 0 $ corresponds to waves within one branch of the homoclinic orbit. Given the symmetry of the phase portraits due to the potential $ V(u) = \frac{1}{4}u^{4} - \frac{1}{2}u^{2} $ being even in $ u $, there is no difference in the spectra of waves chosen within the left or right branch of the homoclinic orbit for equal values of $ E $.	 Without loss of generality we chose to consider waves within the right branch of the homoclinic orbit.
	\begin{figure}[t]
		\centering
		\begin{subfigure}[t]{0.4 \textwidth}
			\captionsetup{justification=centering,margin=1cm}
			\centering
			\includegraphics[width=\linewidth]{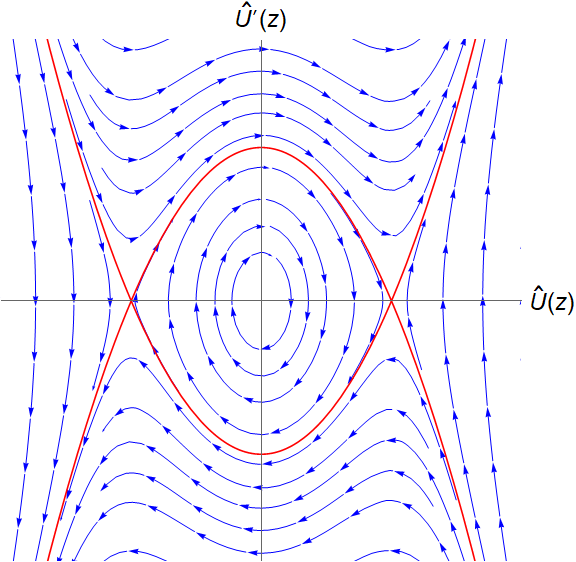}
			\caption{Subluminal $ c^{2} < 1 $.}
			\label{fig:f3a}
		\end{subfigure}
		\begin{subfigure}[t]{0.4 \textwidth}
			\captionsetup{justification=centering,margin=1cm}
			\centering
			\includegraphics[width=\linewidth]{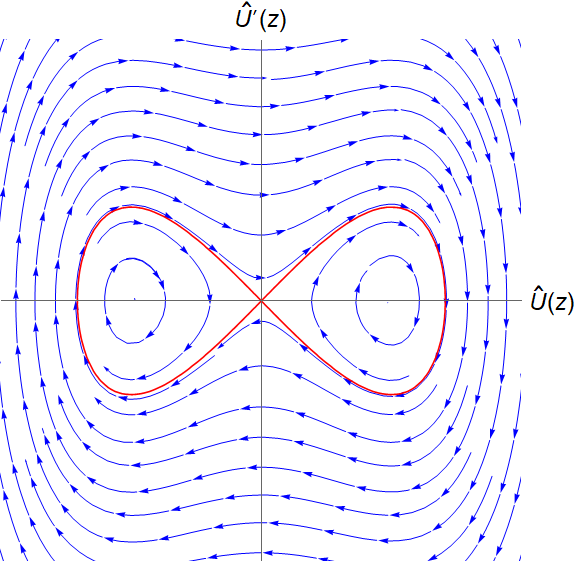}
			\caption{Superluminal $ c^{2} > 1 $.}
			\label{fig:f3b}
		\end{subfigure}
		\caption{Phase portraits for \cref{eq:standwavenlkg} with $ V(\hat{U}) = \frac{1}{4}\hat{U}^{4} - \frac{1}{2}\hat{U}^{2} $. In the subluminal case, the librational periodic solutions $ \hat{U} $ are within the separatrix (drawn in red). For superluminal waves, all periodic waves are librational, and these exist both outside the homoclinic orbit (drawn in red) and inside.}
		\label{fig:f3}
	\end{figure}
	In \cref{fig:f4}, we numerically compute the spectra of several waves using the Evans function $ D_{2}(\zeta;\theta) $. We chose the waves which produced qualitatively different spectra, however we have not proved that this list is exhaustive. We note that all observed waves are spectrally unstable; in the subluminal case, our numerical results agree with the analysis of \cite{Pal2020}.
	\begin{figure}[h!]
		\centering
		\begin{subfigure}[t]{0.32 \textwidth}
			\captionsetup{justification=centering}
			\centering
			\includegraphics[width=\linewidth]{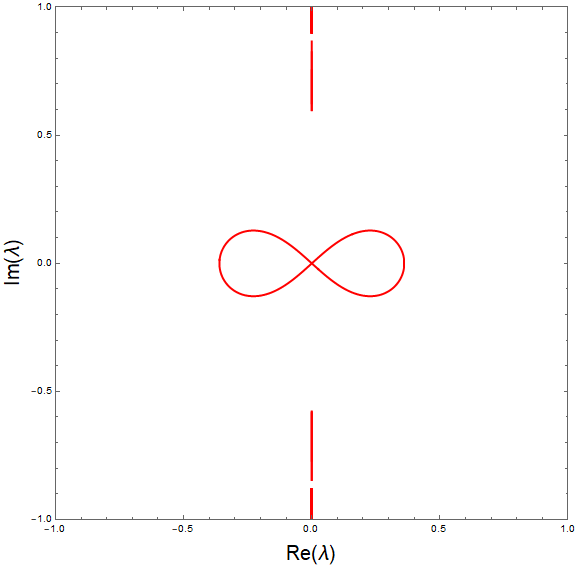}
			\caption{$E = -0.216,\; c = 0.8 $.}
			\label{fig:f4a}
		\end{subfigure}
		\begin{subfigure}[t]{0.32 \textwidth}
			\captionsetup{justification=centering}
			\centering
			\includegraphics[width=\linewidth]{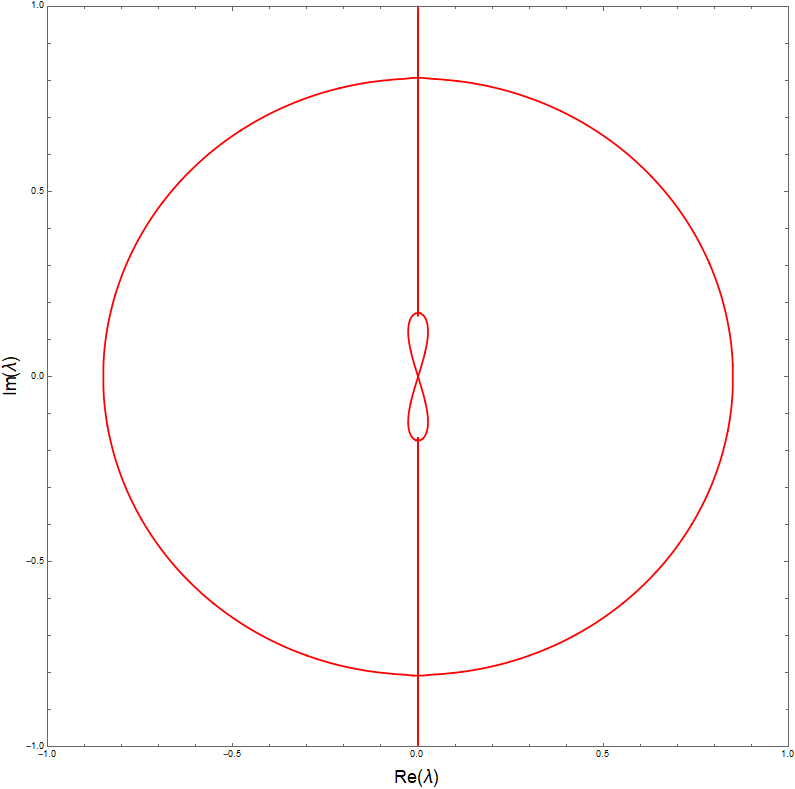}
			\caption{$E = -0.082875,\; c = 0.95 $.}
			\label{fig:f4b}
		\end{subfigure}
		\begin{subfigure}[t]{0.32 \textwidth}
			\captionsetup{justification=centering}
			\centering
			\includegraphics[width=\linewidth]{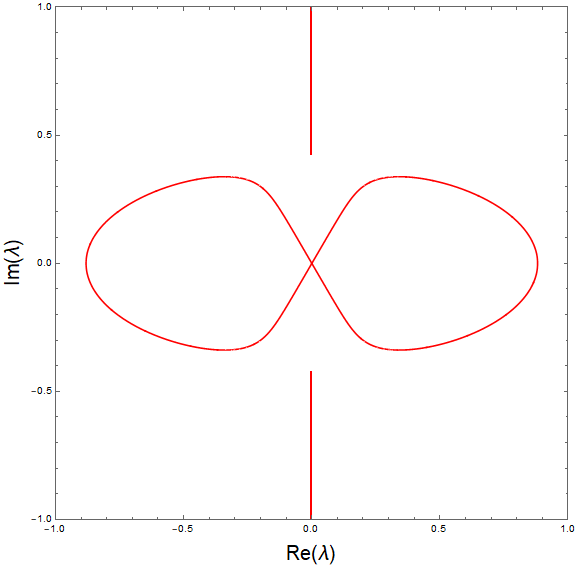}
			\caption{$E = -0.07,\; c = 0.45 $.}
			\label{fig:f4c}
		\end{subfigure}
		\begin{subfigure}[t]{0.32 \textwidth}
			\captionsetup{justification=centering}
			\centering
			\includegraphics[width=\linewidth]{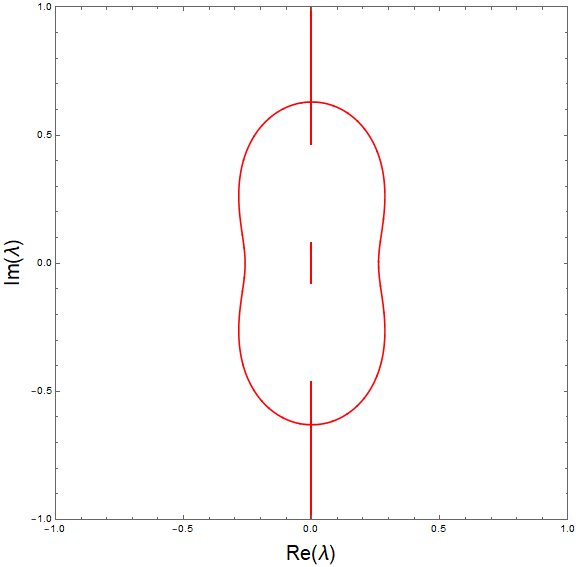}
			\caption{$E = 0.01,\; c = 1.1 $.}
			\label{fig:f4d}
		\end{subfigure}
		\begin{subfigure}[t]{0.32 \textwidth}
			\captionsetup{justification=centering}
			\centering
			\includegraphics[width=\linewidth]{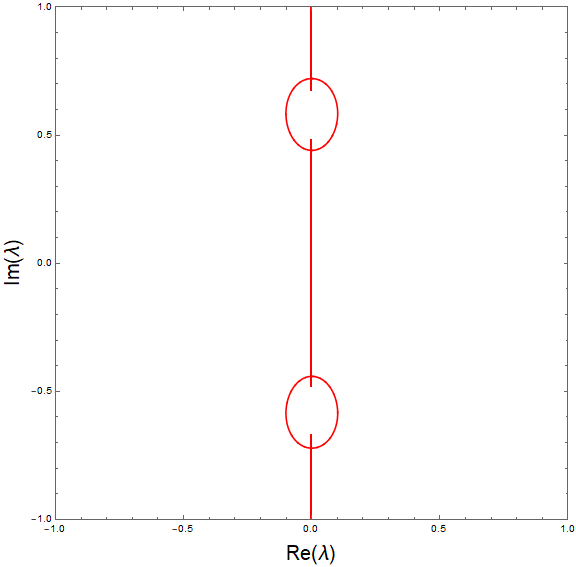}
			\caption{$E = 0.5,\; c = 1.1 $.}
			\label{fig:f4e}
		\end{subfigure}
		\begin{subfigure}[t]{0.32 \textwidth}
			\captionsetup{justification=centering}
			\centering
			\includegraphics[width=\linewidth]{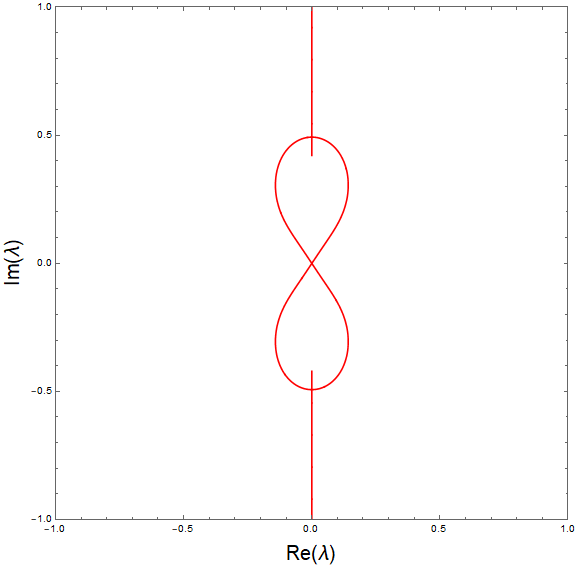}
			\caption{$E = -0.05,\; c = 1.1 $.}
			\label{fig:f4f}
		\end{subfigure}
		\caption{Numerical plots of the spectra $ \sigma $ of various periodic travelling wave solutions to the nonlinear Klein-Gordon \cref{eq:nlKG} with potential $ V(u) = \frac{u^{4}}{4} - \frac{u^{2}}{2} $. Row 1 depicts the spectra for subluminal waves, while row 2 corresponds to superluminal waves.}
		\label{fig:f4}
	\end{figure}
	\Cref{fig:f5} shows the corresponding phase portraits of the waves whose spectra are included in \cref{fig:f4}.
	\begin{figure}[h!]
		\centering
		\begin{subfigure}[t]{0.28 \textwidth}
			\captionsetup{justification=centering}
			\centering
			\includegraphics[width=\linewidth]{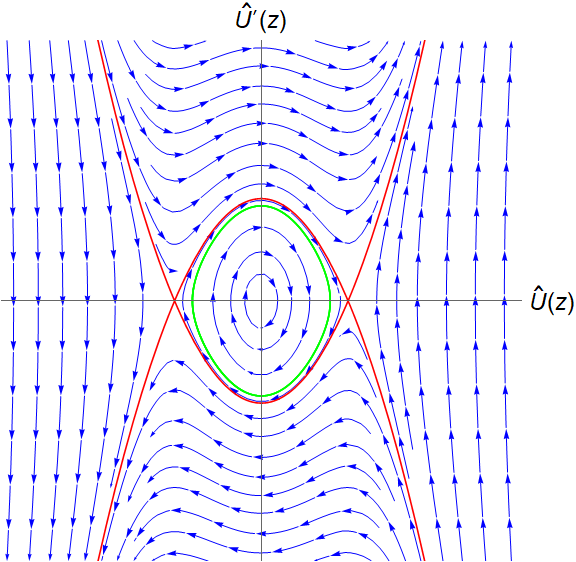}
			\caption{$E = -0.216,\; c = 0.8 $.}
			\label{fig:f5a}
		\end{subfigure}
		\begin{subfigure}[t]{0.28 \textwidth}
			\captionsetup{justification=centering}
			\centering
			\includegraphics[width=\linewidth]{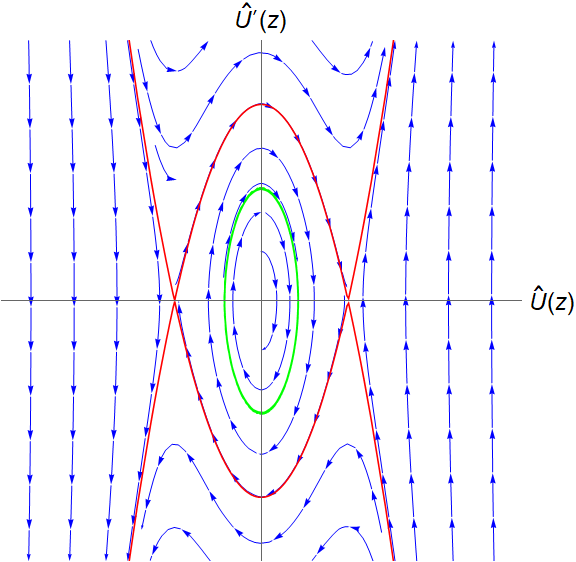}
			\caption{$E = -0.082875,\; c = 0.95 $.}
			\label{fig:f5b}
		\end{subfigure}
		\begin{subfigure}[t]{0.28 \textwidth}
			\captionsetup{justification=centering}
			\centering
			\includegraphics[width=\linewidth]{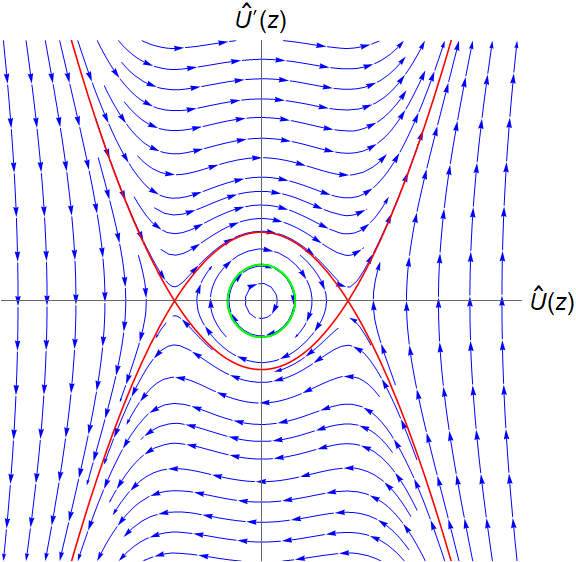}
			\caption{$E = -0.07,\; c = 0.45 $.}
			\label{fig:f5c}
		\end{subfigure}
		\begin{subfigure}[t]{0.28 \textwidth}
			\captionsetup{justification=centering}
			\centering
			\includegraphics[width=\linewidth]{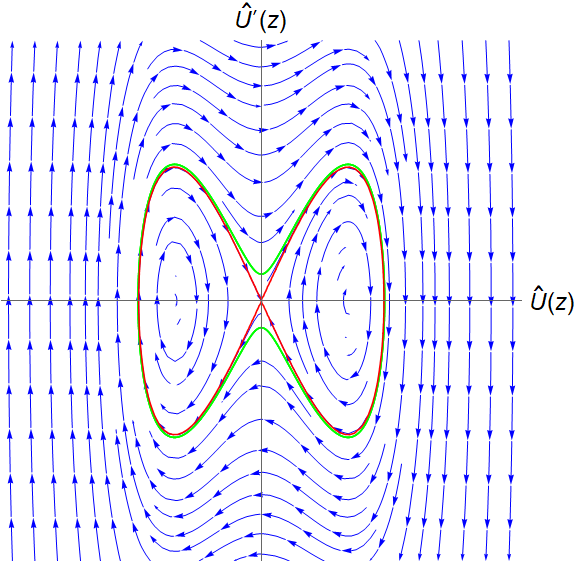}
			\caption{$E = 0.01,\; c = 1.1 $.}
			\label{fig:f5d}
		\end{subfigure}
		\begin{subfigure}[t]{0.28 \textwidth}
			\captionsetup{justification=centering}
			\centering
			\includegraphics[width=\linewidth]{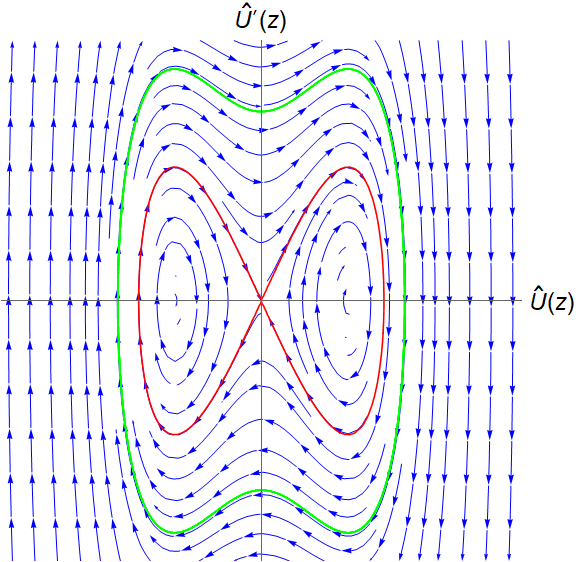}
			\caption{$E = 0.5,\; c = 1.1 $.}
			\label{fig:f5e}
		\end{subfigure}
		\begin{subfigure}[t]{0.28 \textwidth}
			\captionsetup{justification=centering}
			\centering
			\includegraphics[width=\linewidth]{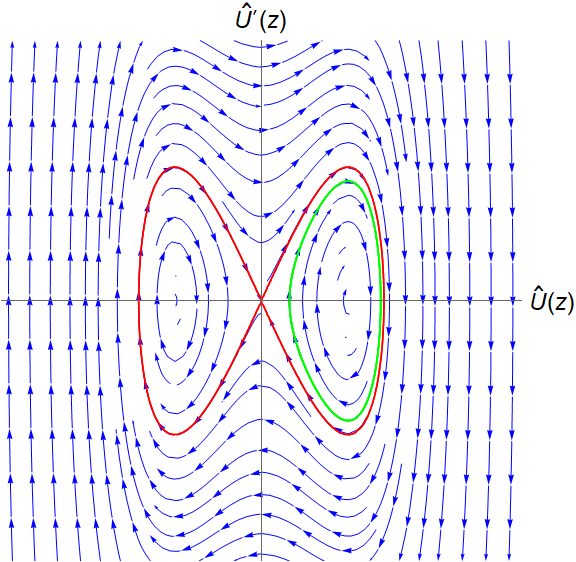}
			\caption{$ E = -0.05,\; c = 1.1 $.}
			\label{fig:f5f}
		\end{subfigure}
		\caption{Phase portraits corresponding to the periodic travelling waves in \cref{fig:f4}. The chosen waves are shown in green, while the separatrices are shown in red.}
		\label{fig:f5}
	\end{figure}
	In \cref{fig:f6}, we capture two bifurcations on the imaginary axis as $ \theta $ is varied for the wave corresponding to \cref{fig:f5b}. We observe a phenomenon discussed in \cite[\S 6]{KM} where characteristic values of opposite Krein signature pass through each other instead of undergoing a Hamiltonian-Hopf bifurcation. A simple characteristic value of Krein signature $ \kappa = -1 $ (denoted by a blue dot) bifurcates onto the real axis in the left plot in the bottom row of \cref{fig:f6}. It passes through the simple characteristic values of Krein signature $ \kappa = 1 $ until it collides with a characteristic value with $ \kappa = 1 $ at $ \zeta \approx 0.8  $, bifurcating off the $ \zeta $-axis. The $ \zeta $ values where the bifurcations take place correspond to the values of $ \lambda $ where the spectrum leaves the imaginary axis, as denoted by the grey arrows. We observed the same phenomenon for the wave whose spectrum is plotted in \cref{fig:f4d}.
	\begin{figure}[h!]
		\centering
		\includegraphics[width=\linewidth]{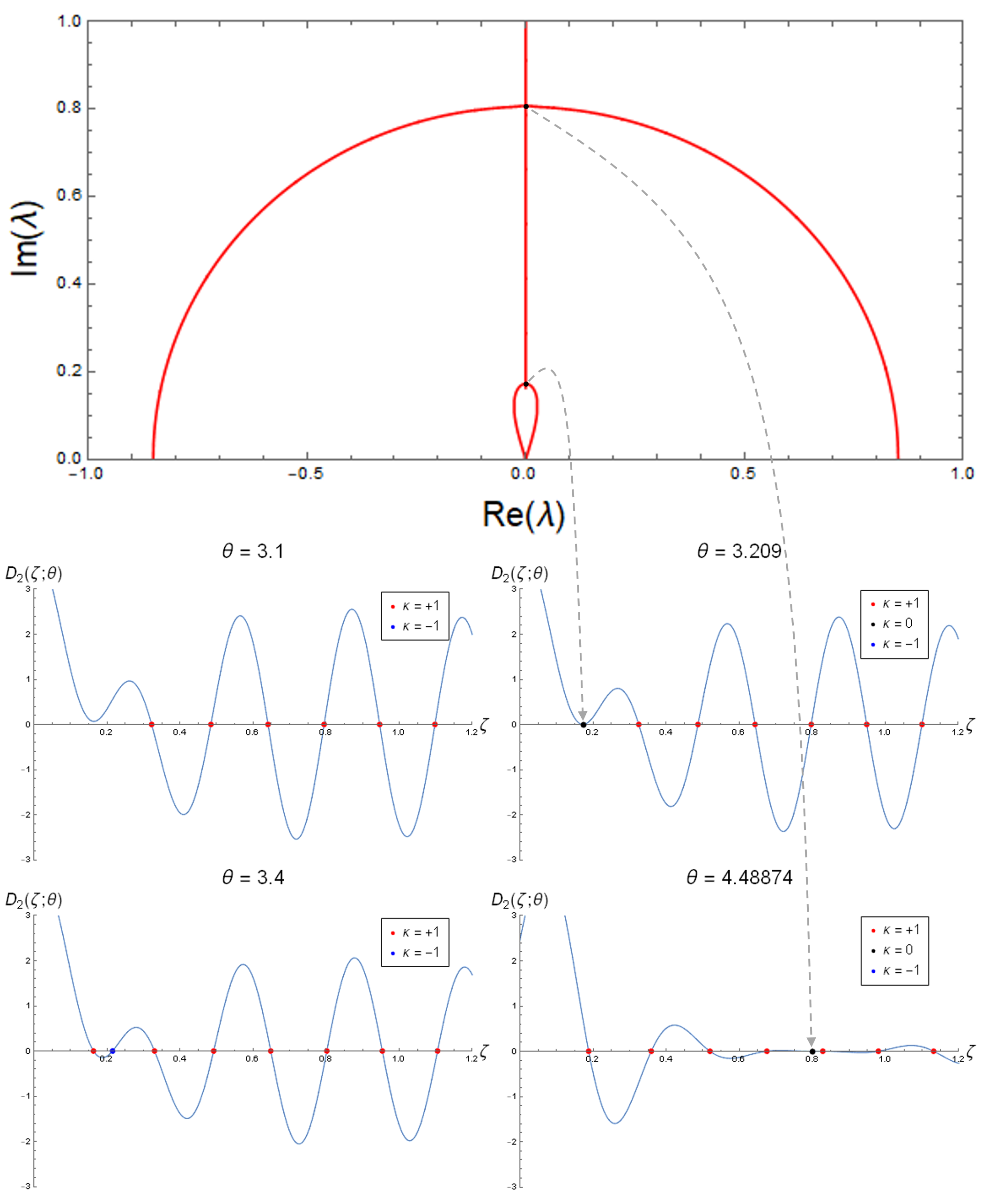}
		\caption{We track the bifurcations in the upper half-plane of the spectrum $ \sigma $ in \cref{fig:f4b} as $ \theta $ is varied. See \cref{fig:f5b} for the original wave. In the middle and bottom rows, we plot the Evans function $ D_{2}(\zeta;\theta) $ for the values: $ \theta = 3.1 $, $ \theta = 3.209 $, $ \theta = 3.4 $, $ \theta = 4.48874 $. The black dots correspond to the values $ \zeta $ where Hamiltonian-Hopf bifurcations occur, and we have labelled these on the spectral diagram in the top row with black dots as well.}
		\label{fig:f6}
	\end{figure}
	\section{Discussion and conclusion}
	In this paper, we apply Floquet theory to results for Hill's equation in \cite{Jones13} to construct a new Evans function for quasi-periodic solutions of the linearised sine-Gordon equation. When compared to the Evans function in \cite{Jones14}, this new Evans function simplifies the calculation of Krein signatures of simple characteristic values when using the Evans-Krein function method. These Krein signatures allow us to track Hamiltonian-Hopf bifurcations in the spectrum of the sine-Gordon equation in terms of the Floquet exponent. This distinguishes our method from \cite{Marangell2015} in which the authors use the Floquet multipliers to develop a criterion for the existence of such bifurcations. As a check on the correctness of our methods, we use this new Evans function to numerically compute spectra for different periodic travelling wave solutions of the sine-Gordon equation, replicating the results of \cite{Jones13,Jones14,Marangell2015}. Finally, as an example of how to extend our Evans function, we use it to compute the spectrum in a general nonlinear Klein-Gordon equation, producing spectral diagrams and calculating Krein signatures for the potential $ V(u) = \frac{1}{4}u^{4} - \frac{1}{2}u^{2} $.
	\section{Acknowledgements}
	The authors would like to thank Peter Miller for his helpful suggestions for improving our paper, Dave Smith for his comments which clarified our notation, and the referees for their helpful suggestions for improving our paper. W.~Clarke would like to thank Henrik Schumacher for his insightful discussion about speeding up our code. R.~Marangell acknowledges the support of the Australian Research Council under grant DP200102130. 
	\bibliographystyle{amsalpha}  


\begin{thebibliography}{1}
		
		\bibitem[AGJ1990]{AGJ1990}
		{\sc J.~Alexander, R.~Gardner and C.~Jones},
		{\em A topological invariant arising in the stability analysis of travelling waves},
		J. reine angew. Math. {\bf 410} (1990), 167--212.
		
		\bibitem[BP1982]{BP}
		{\sc A.~Barone and G.~Patern\`o},
		{\em Physics and Applications of the Josephson Effect},
		John Wiley \& Sons, Inc. (1982).
		
		\bibitem[BEMS1971]{BEM}
		{\sc A.~Barone, F.~Esposito, C.J.~Magee and A.C.~Scott},
		{\em Theory and applications of the sine-gordon equation},
		La Rivista del Nuovo Cimento {\bf 1} (1971), no. 2, 227--267.
		
		\bibitem[BJK2011]{BJK2011}
		{\sc J.C.~Bronski, M.A.~Johnson and T.~Kapitula},
		{\em An index theorem for the stability of periodic travelling waves of Korteweg-de Vries type},
		Proc. Roy. Soc. Edinburgh: Sect. A {\bf 141} (2011), no. 6, 1141--1173.
		
		\bibitem[BJK2014]{BJK2014}
		{\sc J.C.~Bronski, M.A.~Johnson and T.~Kapitula},
		{\em An instability index theory for quadratic pencils and applications},
		Comm. Math. Phys. {\bf 327} (2014), no. 2, 521--550.
		
		\bibitem[CFMT2019]{CFMT}
		{\sc M.~Cadoni, E.~Franzin, F.~Masella and M. Tuveri},
		{\em A Solution-Generating Method in Einstein-Scalar Gravity},
		Acta Applicandae Mathematicae {\bf 162} (2019), no. 1, 33--45.
		
		\bibitem[DDKS2012]{DDKS}
		{\sc G.~Derks, A.~Doelman, J.K.~Knight and H.~Susanto},
		{\em Pinned fluxons in a Josephson junction with a finite-length inhomogeneity},
		European Journal of Applied Mathematics {\bf 23} (2012), no. 2, 201--244.
		
		\bibitem[DDGV2003]{DDGV}
		{\sc G.~Derks, A.~Doelman, S.A.~van Gils and T.~Visser},
		{\em Travelling waves in a singularly perturbed sine-Gordon equation},
		Physica D: Nonlinear Phenomena {\bf 180} (2003), no. 1-2, 40--70.
		
		\bibitem[DG2011]{DG}
		{\sc G.~Derks and G.~Gaeta},
		{\em A minimal model of DNA dynamics in interaction with RNA-Polymerase},
		Physica D: Nonlinear Phenomena {\bf 240} (2011), no. 22, 1805--1817.
		
		\bibitem[Eva1972]{Evans}
		{\sc J.W.~Evans},
		{\em Nerve Axon Equations: III Stability of the Nerve Impulse},
		Indiana University Mathematics Journal {\bf 22} (1972), no. 6, 577--593.
		
		\bibitem[Gar1993]{Gar1993}
		{\sc R.A.~Gardner},
		{\em On the structure of the spectra of periodic travelling waves},
		J. Math. Pures Appl. {\bf 72} (1993), no. 5, 415--439.
		
		\bibitem[Gar1997]{Gar1997}
		{\sc R.A.~Gardner},
		{\em Spectral analysis of long wavelength periodic waves and applications},
		J. Reine Angew. Math. {\bf 491} (1997), 149--181.
		
		\bibitem[HK2008]{HK2008}
		{\sc M.~H\v{a}r\v{a}gu{\c{s}} and T.~Kapitula},
		{\em On the spectra of periodic waves for infinite-dimensional Hamiltonian systems},
		Physica D: Nonlinear Phenomena {\bf 237} (2008), no. 20, 2649--2671.
		
		\bibitem[Jones1984]{Jones84}
		{\sc C.K.R.T.~Jones},
		{\em Stability of the travelling wave solution of the FitzHugh-Nagumo system},
		Transactions of the American Mathematical Society {\bf 286} (1984), no. 2, 431--469.
		
		\bibitem[JMMP2013]{Jones13}
		{\sc C.K.R.T.~Jones, R.~Marangell, P.D.~Miller and R.G.~Plaza},
		{\em On the stability analysis of periodic sine-Gordon traveling waves},
		Physica D: Nonlinear Phenomena {\bf 251} (2013), 63--74.
		
		\bibitem[JMMP2014]{Jones14}
		{\sc C.K.R.T.~Jones, R.~Marangell, P.D.~Miller and R.G.~Plaza},
		{\em Spectral and modulational stability of periodic wavetrains for the nonlinear Klein-Gordon equation},
		Journal of Differential Equations {\bf 257} (2014), no. 12, 4632--4703.
		
		\bibitem[Kap2010]{Kap2010}
		{\sc T.~Kapitula},
		{\em The Krein signature, Krein eigenvalues, and the Krein oscillation theorem},
		Indiana Univ. Math. J. {\bf 59} (2010), no. 4, 1245--1275.
		
		\bibitem[KKY2013]{KKY2013}
		{\sc T.~Kapitula, P.G.~Kevrekidis and D.~Yan},
		{\em The Krein matrix: general theory and concrete applications in atomic Bose-Einstein condensates},
		SIAM J. Appl. Math. {\bf 73} (2013), no. 4, 1368--1395.
		
		\bibitem[Kat1976]{Kato}
		{\sc T.~Kato},
		{\em Perturbation Theory for Linear Operators},
		Berlin: Springer (1976).
		
		\bibitem[Kno2000]{Knobel}
		{\sc R.~Knobel},
		{\em An Introduction to the Mathematical Theory of Waves}, American Mathematical Society (2000), Providence, R.I.
		
		\bibitem[KDT2019]{KDT2019}
		{\sc R.~Koll{\'a}r, B.~Deconinck and O.~Trichtchenko},
		{\em Direct characterization of spectral stability of small-amplitude periodic waves in scalar Hamiltonian problems via dispersion relation},
		SIAM J. Math. Anal. {\bf 51} (2019), no. 4, 3145--3169.
		
		\bibitem[KM2014]{KM}
		{\sc R.~Koll{\'a}r and P.D.~Miller},
		{\em Graphical Krein Signature Theory and Evans-Krein Functions},
		SIAM Review {\bf 56} (2014), no. 1, 73--123.
		
		\bibitem[LLM2011]{Lafortune}
		{\sc S.~Lafortune, J.~Lega and S.~Madrid},
		{\em Instability of Local Deformations of an Elastic Rod: Numerical Evaluation of the Evans Function},
		SIAM Journal on Applied Mathematics, \textbf{71} (2011), no. 5, 1653--1672.
		
		\bibitem[MW2013]{MW}
		{\sc W.~Magnus and S.~Winkler},
		{\em Hill's equation},
		Courier Corporation (2013).
		
		\bibitem[MM2015]{Marangell2015}
		{\sc R.~Marangell and P.D.~Miller},
		{\em Dynamical Hamiltonian-Hopf instabilities of periodic traveling waves in Klein-Gordon equations},
		Physica D: Nonlinear Phenomena {\bf 308} (2015), 87--93.
		
		\bibitem[Mar1988]{Markus1988}
		{\sc A.S.~Markus},
		{\em Introduction to the spectral theory of polynomial operator pencils},
		American Mathematical Society (1988), Providence, R.I.
		
		\bibitem[Nat2011]{Natali2011}
		{\sc F.~Natali},
		{\em On periodic waves for sine- and sinh-Gordon equations},
		J. Math. Anal. Appl. {\bf 379} (2011), no. 1, 334--350.
		
		\bibitem[Pal2020]{Pal2020}
		{\sc J.M.~Palacios},
		{\em Orbital stability and instability of periodic wave solutions for the $\phi^4$-model}, Preprint (2020), arXiv:2005.09523.
		
		\bibitem[PN2014]{PN2014}
		{\sc J.A.~Pava and F.~Natali},
		{\em (Non)linear instability of periodic traveling waves: Klein-Gordon and KdV type equations},
		Advances in Nonlinear Analysis {\bf 3} (2014), no. 2, 95--123.
		
		\bibitem[PP2016]{PP2016}
		{\sc J.A.~Pava and R.G.~Plaza},
		{\em Transverse orbital stability of periodic traveling waves for nonlinear Klein-Gordon equations},
		Stud. Appl. Math. {\bf 137} (2016), no. 4, 473--501.
		
		\bibitem[Sco1969]{Scott}
		{\sc A.C.~Scott},
		{\em Waveform Stability on a Nonlinear Klein-Gordon Equation},
		Proceedings of the IEEE, \textbf{57} (1969), no. 7, 1338--1339.
		
		\bibitem[SS2012]{Stanislavova}
		{\sc M.~Stanislavova and A.~Stefanov},
		{\em Linear stability analysis for travelling waves of second order in time PDE's},
		Nonlinearity, \textbf{25} (2012), no. 9, 2625--2654.
		
		\bibitem[TDK2018]{TDK2018}
		{\sc O.~Trichtchenko, B.~Deconinck and R.~Koll{\'a}r},
		{\em Stability of periodic traveling wave solutions to the Kawahara equation},
		SIAM J. Appl. Dyn. Syst. {\bf 17} (2018), no. 4, 2761--2783.
	\end{thebibliography}

\end{document}